\documentclass[11pt]{amsart}%fleqn
\usepackage{fullpage}

\usepackage{amsmath}
\usepackage{amssymb}
\usepackage{amsfonts}
\usepackage{graphicx}
\usepackage{amsthm}
\usepackage{enumerate}
\usepackage{lscape}
\usepackage{dsfont}
\usepackage{color}
\usepackage{mathtools}
\usepackage[dvipsnames]{xcolor}
\usepackage{hyperref}

\newcounter{example}
\newenvironment{example}[1][]{\refstepcounter{example}\par\medskip
   \noindent \textbf{Example~\theexample. (#1)} \rmfamily}{\medskip}    

\usepackage[utf8]{inputenc}

\newcommand{\de}{\partial}          
\newcommand{\deb}{\overline\partial}  

\newcommand{\C}{\mathds C}
\newcommand{\R}{\mathds R}  

\newcommand{\rank}{\operatorname{rank}}

\newcommand{\tr}{\operatorname{tr}}

\newcommand{\ov}[1]{\overline{#1}}
\newcommand{\wi}[1]{\widetilde{#1}}

\newcommand{\f}{\rightarrow}
\newcommand{\Aut}{\operatorname{Aut}}

\newcommand{\K}{K\"ahler}

\newcommand{\lmb}{\lambda}

\newcommand{\nax}[1]{\left\|{#1}\right\|_{\max}}

\newcommand{\di}{{\operatorname{d}}}

\newcommand{\CH}{Cartan--Hartogs\ }
\newcommand{\W}{\Omega}
\newcommand{\SW}{{\W^*}}
\newcommand{\w}{\omega}
\newcommand{\sw}{\omega^*}

\newcommand{\Ns}{N^\mu_{\Omega^*}}

\newcommand{\midd}{\big |}

\newcommand{\midddd}{\bigg |}

\usepackage{apptools}

\newtheorem{theorem}{Theorem}

\newtheorem{lemma}{Lemma}
\newtheorem{prop}{Proposition}

\newtheorem{remark}{Remark}

\newtheorem{question}{Q}

\makeatletter
\@namedef{subjclassname@2020}{%
  \textup{2020} Mathematics Subject Classification}
\makeatother

\theoremstyle{definition}

\makeatletter
\def\@tocline#1#2#3#4#5#6#7{\relax
  \ifnum #1>\c@tocdepth % then omit
  \else
    \par \addpenalty\@secpenalty\addvspace{#2}%
    \begingroup \hyphenpenalty\@M
    \@ifempty{#4}{%
      \@tempdima\csname r@tocindent\number#1\endcsname\relax
    }{%
      \@tempdima#4\relax
    }%
    \parindent\z@ \leftskip#3\relax \advance\leftskip\@tempdima\relax
    \rightskip\@pnumwidth plus4em \parfillskip-\@pnumwidth
    #5\leavevmode\hskip-\@tempdima
      \ifcase #1
       \or\or \hskip 1em \or \hskip 2em \else \hskip 3em \fi%
      #6\nobreak\relax
    \hfill\hbox to\@pnumwidth{\@tocpagenum{#7}}\par% <---- \dotfill -> \hfill
    \nobreak
    \endgroup
  \fi}
\makeatother
\makeatletter

\numberwithin{equation}{section}

\def\subsubsection{\@startsection{subsubsection}{3}%
\z@{.5\linespacing\@plus.7\linespacing}{-.5em}%
{\normalfont\bfseries}}
\makeatother

\title{Symplectic geometry of Cartan--Hartogs domains}

\author{Roberto Mossa}
\address{(Roberto Mossa) Departamento de Matemática \\
 Instituto de Matemática e Estatística \\
Universidade de São Paulo (Brazil)}
        \email{robertom@ime.usp.br}
\author{Michela Zedda}
\address{(Michela Zedda) Dipartimento di Scienze Matematiche, Fisiche e Informatiche \\
         Universit\`a di Parma (Italy)}
\email{michela.zedda@unipr.it}
\date{\today}

\thanks{
The first-named author was supported by a grant from Fapesp (2018/08971-9), the second-named author has been financially supported by the Programme “FIL-Quota Incentivante” of University of Parma, co-sponsored by Fondazione Cariparma.} 

\begin{document}

\maketitle

\begin{abstract}
This paper studies the geometry of Cartan--Hartogs domains from the symplectic point of view. Inspired by duality between compact and noncompact Hermitian symmetric spaces, we construct a dual counterpart of Cartan--Hartogs domains and give explicit expression of global Darboux coordinates for both Cartan--Hartogs and their dual. Further, we compute their symplectic capacity and show that a Cartan--Hartogs admits a symplectic duality if and only if it reduces to be a complex hyperbolic space.  
\end{abstract}
\tableofcontents

\section{Introduction and statement of the results}

Studying the symplectic geometry of a domain $X\subset\C^k$ equipped with a real analytic \K\ metric $\w=\frac{i}{2}\de\deb\phi$, the following questions naturally arise:
\begin{question}\label{quno}
There exists global symplectic coordinates for $X$?
\end{question}
\begin{question}\label{qdue}
Is the dual domain $X^*$ of $X$ a well-defined K\"ahler manifold?
\end{question}
When Q \ref{qdue} has a positive answer we also have: 
\begin{question}\label{qquattro}
What can we say about the symplectic capacity of $X$ and $X^*$?
\end{question}
\begin{question}\label{qtre}
Is there a symplectic duality between $X$ and $X^*$?
\end{question}
Recall that the existence of local symplectic coordinates is guaranteed by the celebrated Darboux Theorem, but in general the answer to Q \ref{quno} is negative, as shown by Gromov's exotic symplectic structures on $\mathds R^{2k}$ \cite{gromovexotic} (see also \cite{bates} for an explicit example of a symplectic manifold diffeomorphic but not symplectomorphic to $\mathds R^{4}$).

The concept of duality in Q \ref{qdue} and Q \ref{qtre} is inspired by the natural duality between compact and noncompact hermitian symmetric spaces and can be expressed as follows. 
The potential $\phi$ can be expanded as a power series of the variables $z=\left(z_{1}, \ldots, z_{k}\right)$ and $\bar{z}=\left(\bar{z}_{1}, \ldots, \bar{z}_{k}\right),$ denoted by $\phi(z,\bar{z})$,  where $z$ is the restriction to $X$ of the Euclidean coordinates of $\C^{k}$. By the change of variables $\bar{z} \mapsto-\bar{z}$ in this power series one gets a new power series denoted by $\phi(z,-\bar{z}).$ We say (according to \cite{lmbdual}) that a symplectic manifold $\left(X ^* , \w^* \right)$ is the \emph{symplectic dual} of $\left(X  , \w \right)$ if $\w^*$ has a \K\ potential $\phi^*$ such that the power series $\phi^*(z,\bar{z})$ associated to $\phi^*$ formally satisfies:
$$ 
\phi^*(z,\bar{z})=-\phi(z,-\bar{z}).
$$
%This definition should be compared to \cite{lmbdual}, where the existence of dual domains of $\mathds C^k$ whose K\"ahler potential depends only the module of the variables has been investigated.  
We say that a smooth map $\psi\!:X\rightarrow  X^*$  is a \emph{symplectic duality} if it satisfies:
$$
\psi^*\omega_0=\omega \quad \text{ and } \quad \psi^*\omega^*=\omega_0,
$$
where we denote by $\w_0 = \frac{i}{2} \sum_{j=1}^k dz_j\wedge d\ov z_j$ the restriction of the flat form of $\C^k$ to $X$ and $X^*$.

Symplectic capacities are a class of symplectic invariants, which generalizes the concept of Gromov width, defined by I. Ekeland and H. Hofer \cite{ekelandhoferI,ekelandhoferII} for domains in $\mathds R^{2n}$ and generalized by H. Hofer and E. Zehnder to symplectic manifolds  \cite{HOFERZEHNDER90} (we refer the reader to Section \ref{capacitysection} for definitions and to \cite{hoferzhenerbook} and references therein for further details). Symplectic capacities naturally represent an obstruction for the existence of a symplectic embedding, as they generalize the concept of Gromov width introduced in \cite{GROMOV85}, which gives a measure of the largest ball that can be symplectically embedded inside a symplectic manifold. Their importance arises in the celebrated Gromov's nonsqueezing Theorem, according to which a symplectic embedding of a ball into a cylinder is possible if only if the ray of the ball is less or equal the cylinder's one. Computations and estimates of the Gromov width and the Hofer--Zehnder capacity can be found e.g. in \cite{BIRAN99,caviedes,GROMOV85,JIANG00,GWgrass,hwasuh,LAMCSC11,LU06,MCDUFF94,SCHLENK05,GWcoadjoint}.

%In literature slightly different definitions are given, for us a symplectic capacity is a map $c\!:{\mathcal C} (2n)\rightarrow [0,+\infty]$, where ${\mathcal C} (2n)$ is the class of all symplectic manifolds of dimension $2n$, satisfying the monotonicity, conformality and nontriviality conditions (see Section \ref{capacitysection}).
%One way to define a symplectic capacity is in terms of symplectic embeddings, as for example, the Gromov width $c_G(X,\omega)$ (introduced in \cite{GROMOV85}). Let $a$ be the supremum over the rays $r$  such that the ball $B(r)$ can be symplectically embedded into $(X,\omega)$,  the Gromov width is defined as $c_G(X,\omega)=a^2\pi$. This trivially satisfies monotonicity and conformality, while the nontriviality property comes from the celebrated \emph{Gromov's
%nsqueezing Theorem}, according to which  the existence of a symplectic  embedding of the ball $B(a)$ into a cylinder $Z(r)$ of ray $r$ implies  $a\leq r$.
%Furthermore, it is not hard to see that the existence of any capacity implies 
%Gromov's nonsqueezing Theorem. Computations and estimates of the Gromov width can be found e.g. in \cite{BIRAN99,GROMOV85,JIANG00,GWgrass,LAMCSC11,LU06,MCDUFF94,SCHLENK05,GWcoadjoint}.
%

All the fourth questions find a positive answer when $X$ is a Hermitian symmetric space of noncompact type $\Omega$. In particular in \cite{loiscala}, global symplectic coordinates that realize a symplectic duality are given, while in \cite{lmz} is computed the symplectic capacity of $\Omega$, the Gromov width of the dual  $\Omega^*$ and also sharp estimations of the Hofer-Zehnder capacity of $\Omega^*$, extending G. Lu's results  in \cite{LU06} valid for Grassmanians.
It is then natural to investigate if these results can be extended to a class of bounded domains on $\mathds C^{n+1}$ called {\em Cartan--Hartogs domains}, which are Hartogs domains based on hermitian symmetric spaces of noncompact type. The geometry of these domains turned out to be interesting from several points of view, see e.g. \cite{fengtu,articwall,balancedch,berezinCH} and references therein (see also \cite{Ishi} for Hartogs domains constructed over bounded homogeneous domains). More precisely, Cartan-Hartogs domains are a $1$-parameter family of noncompact nonhomogeneous domains of $\mathds C^{n+1}$, given by:
\begin{equation}\label{defCHdom}\begin{split} 
 M_{\Omega,\mu}:=\{(z,w)\in \Omega\times \C \mid |w|^2<N^\mu_{\W}(z, \ov z)\},
\end{split}\end{equation}
where $\Omega\subset\mathds C^n$ is a bounded symmetric domain, $N^\mu_{\W}(z,\ov z)$ is its generic norm and  $\mu>0$ is a positive real parameter. We endow $M_{\W,\mu}$ with its Kobayashi \K\ form $\omega_{\W,\mu}$  whose global K\"ahler potential reads:
\begin{equation}\label{phicartan}
\varphi_{\W,\mu}(z,w;\bar z,\bar w)=-\log(N^\mu_{\W}(z,\ov z)-|w|^2).
\end{equation}
Our first result answer positively to Q \ref{quno} when $X$ is a Cartan--Hartogs domain:
\begin{theorem}\label{thmmain} 
Let $M_{\W,\mu}$ be an $(n+1)$-dimensional  Cartan-Hartogs domain. Then there exists a global symplectomorphism $\Psi_{\W,\mu}\!:M_{\W,\mu} \rightarrow \C^{n+1}$.
\end{theorem}
The global coordinates we exhibit generalize those of the hermitian symmetric space of noncompact type $\Omega$ the Cartan--Hartogs is based on. Indeed, if we restrict the map $\Psi_{\Omega,\mu}$ to the base $\W$, we obtain  
the symplectic coordinates for $\Omega$ constructed by A. Di Scala and A. Loi in  \cite{loiscala} (see also \cite{bisy}). Further, we prove that as well as Di Scala and Loi's map, $\Psi_{\Omega,\mu}$ well-behaves with respect to the action of the automorphism group of  $\Omega$ and enjoys a nice hereditary property (see Remarks \ref{heredrem} and \ref{autrem}).

We construct (see Lemma \ref{lemswkahl} below) a symplectic dual of Cartan--Hartogs domains, that is, $M^*_{\W,\mu}=\C^{n+1}$ equipped with the dual \K\ form $\omega^*_{\W,\mu}$, that we show to be strictly plurisubharmonic on $\mathds C^{n+1}$. In a natural way, the dual counterpart of $\Psi_{\W,\mu}$ defines global Darboux coordinates for $M^*_{\W,\mu}$, and we have the following:

\begin{theorem}\label{thmmaindual} 
The dual Cartan-Hartogs domain $(M^*_{\W,\mu},\omega^*_{\W,\mu})$ is a well-defined K\"ahler manifold which admits global Darboux coordinates $\Phi_{\W,\mu}\!:M^*_{\W,\mu} \rightarrow \C^{n+1}$.
\end{theorem}

We prove also that $\Phi_{\W,\mu}$ is compatible with the action of the automorphism group of  $\Omega$ and enjoys of an hereditary property analogously to $\Psi_{\Omega,\mu}$ (see Remark \ref{rmkphi}).

As third result we compute the symplectic capacity of $\left(M_{\Omega,\mu}, \w_0 \right)$ and $\left(M^*_{\Omega,\mu},\w^*_{\W,\mu}\right)$, answering for these domains to Q \ref{qquattro}:
\begin{theorem}\label{propsimpc}%\label{thmdualcap}
Let $c$ be a symplectic capacity. Then for a \CH domain $M_{\Omega,\mu}$ equipped with the flat form $\w_0$, and for its dual $M_{\Omega,\mu}$ endowed with the dual form $\w^*_{\W,\mu}$, one has:
\begin{equation*}\begin{split} 
 c\left(M_{\Omega,\mu}, \w_0 \right) = \pi,\quad \text{if } \mu\leq 1;
\end{split}\end{equation*}
\begin{equation*}
 c\left(M^*_{\Omega,\mu},\w^*_{\W,\mu}\right) =\left\{\begin{array}{ll}
\mu^2\pi   &\text{if }\ \mu < 1,\\
 \pi  & \text{if }\ \mu \geq 1.
\end{array}\right.
\end{equation*}
\end{theorem}
The proof of the first part of Theorem \ref{propsimpc} is based on the results in \cite{lmz}, on the symplectic capacity of Hermitian symmetric spaces of  noncompact type.
To prove the second part we apply Theorem \ref{thmmaindual}.

Unfortunately, it can be proven that $\Psi_{\Omega,\mu}$ of Theorem \ref{thmmain} is not a symplectic duality unless the Cartan--Hartogs reduces to be a complex hyperbolic space, i.e. when $\Omega=\mathds C{\rm H}^n$ and $\mu=1$. With Theorem \ref{mainduality} below we show that this is not a peculiarity of our map, giving a negative answer to Q \ref{qtre} that characterizes the complex hyperbolic space among Cartan--Hartogs domains:
\begin{theorem}\label{mainduality}
There exists a symplectic duality between a Cartan--Hartogs domain $\left(M_{\W,\mu}, \w
_{\W,\mu}\right)$ and its dual $\left(\C^{n+1}, \w^*
_{\W,\mu}\right)$ if and only if $\left(M_{\W,\mu}, \w_{\W,\mu}\right)=\left(\C{\rm H} ^{n+1}, \w_{hyp}\right)$. This is equivalent to $\Psi_{\W,\mu}=\Phi^{-1}_{\W,\mu}$, and in this case $\Psi_{\W,\mu}$ realize a symplectic duality. 
\end{theorem}

The proof is obtained when $\mu<1$ as direct consequence of Theorem \ref{propsimpc} while for $\mu\geq 1$  it is a consequence of a volume comparison.\\

The paper is organized as follows. In the next section we describe the geometry of Cartan--Hartogs domains, proving Theorem \ref{thmmain}. Sections \ref{dualsection} is devoted to the construction of dual Cartan--Hartogs and the proof of Theorem \ref{thmmaindual}. Finally in Section \ref{capacitysection} we prove Theorem \ref{propsimpc} and Theorem \ref{mainduality}. \\

The authors are grateful to Prof. Andrea Loi for his interest in their work and for the useful comments. 

\section{Cartan--Hartogs domains and the proof of Theorem \ref{thmmain}}\label{giacomo}
Throughout this section we use the Jordan triple system theory, referring the reader  to \cite{loiscala,lm1,lmz,bisy,loos,m1,roos} for details and further applications.

\subsection{Definition and geometric properties}
Consider an Hermitian  symmetric space of noncompact type (from now on HSSNT) $\W$ and the associated Hermitian  positive Jordan triple systems (from now on HPJTS) $\left({V}, \{ ,  ,\}\right)$ (see e.g. \cite[Section 2.2]{loiscala}). Recall that there is a natural identification between $V$ equipped with the flat form $\omega_0:=\frac{i}{2} \partial \bar{\partial} m_{1}(x, \ov x)$, where $m_1$ is the generic trace of $V$, and $\C^n$ equipped with the standard flat form $\w_0=\sum_{j=1}^n dz_j\wedge d\ov z_j$. By mean of this identification, from now on we will always consider $\W$ as a bounded symmetric domain of $\C^n$ in its (unique up to linear isomorphism) circled realization, which is usually called a Cartan domain when $\W$ is irreducible. 
Analogously, we will consider the Bergman operator $\operatorname{B}_\Omega$ as operator on $\C^n$, and its generic norm $N_\Omega(z,\ov z)$ as a polynomial of $\C^n$ (see e.g. \cite[Section 2.1]{loiscala}).

To any HSSNT $\W$ we can associate the \emph{Cartan--Hartogs domain} $M_{\W,\mu}$, defined in \eqref{defCHdom}, equipped with its Kobayashi metric:
\begin{equation}\label{kobo}\begin{split} 
 \omega_{\Omega,\mu}=\frac{i}{2}\partial\bar\partial \varphi_{\W,\mu},
\end{split}\end{equation}
where $\varphi_{\W,\mu}$ is the \K\ potential defined in \eqref{phicartan}. Notice that if we restrict the Kobayashi metric $\omega_{\Omega,\mu}$ to $\left\{(z,0)\in M_\W\right\}\cong\W$ we get a multiple of the hyperbolic form:
\begin{equation}\label{eqwhyp}
\omega_{hyp}:=-\frac i2\partial\bar \partial \log N_\Omega(z,\bar z),
\end{equation}
 of $\W$ (see also  \cite[(11)]{loiscala}), i.e. $\omega_{{\Omega,\mu}_{\mid_{\W}}}={\mu}\,\w_{hyp}$.\\

 Define:
\begin{equation}\label{francesca}
F(s)=\frac1{2^rr!}\prod_{j=1}^r\frac{\Gamma(b+1+(j-1)\frac a2)\Gamma(s+1+(j-1)\frac a2)\Gamma(j\frac a2+1)}{\Gamma(s+b+2+(r+j-2)\frac a2)\Gamma(\frac a2+1)},
\end{equation}
for $a$, $b$ the numerical invariants of $\Omega$, $r$ its rank and $s\in \R^+$. In  \cite[Prop. 2.1]{roosvolume} W. Yin, K. Lu and G. Roos prove that:
\begin{equation*}%\label{volumeomegaf}
\int_\Omega N(z,\ov z)^s\w_0^n=\pi^nF(s)\int_{\mathcal F}\Theta,%=\frac{F(s)}{F(0)}\int_\Omega\w_0^n.
\end{equation*}
 where $\Theta$ is the induced volume form on F\"urstenberg-Satake boundary $\mathcal{F}$ of $\W$ (see e.g. \cite[(1.28)]{dilr} for its definiton).
Thus:
\begin{equation*}
\begin{split}
\operatorname {Vol}(M_{\W,\mu},\omega_0)=\int_{M_{\W,\mu}}\frac{\omega_0^{n+1}}{(n+1)!}=\pi\int_\Omega \int_0^{N^\mu}dr_w\wedge \frac{\omega_0^n}{n!}=\frac{\pi}{n!}\int_\Omega N^\mu \w_0^n= \frac{\pi^{n+1} }{n!}F(\mu) \int_\mathcal{F} \Theta,%= \frac{\pi }{n!}\frac{F(\mu) }{F(0)}\int_{\Omega} \w_0^n.
\end{split}
\end{equation*}
i.e. the volume of $M_{\W,\mu}$ with respect to the flat form induced by $\C^{n+1}$ is given by:
\begin{equation}\label{lemvolmw0}
\begin{split}
\operatorname {Vol}(M_{\W,\mu},\omega_0)= \frac{\pi^{n+1} }{n!}F(\mu) \int_\mathcal{F} \Theta.
\end{split}
\end{equation}

The following is a key example for our analysis:
\begin{example}[Hartogs--polydisc]\label{hartogspoly}
Let $\Delta^n$ be the $n$-dimensional polydisc:
$$
\Delta^n = \left\{z=\left(z_1,\dots, z_n\right)\in \C^n \mid |z_j|^2<1, \, j=1,\dots,n \right\},
$$
the generic norm is given by (Hua \cite{hua}):
$$
N_{\Delta^n}(z, \ov  z) = \prod_{j=1}^n(1-|z_j|^2).
$$ 
It follows that its hyperbolic metric reads:
$$
\omega_{hyp}= -\frac{i}{2}\de\deb \log \left(\prod_{j=1}^n(1-|z_j|^2)\right),
$$
and that the associated Cartan-Hartogs domain, which we call {\em Hartogs--polydisc}, is given by:
$$
M_{\Delta^n,\mu}=\left\{(z,w) \mid z\in \Delta ^n,\, |w|^2 < \prod_{j=1}^n(1-|z_j|^2)^\mu\right\},
$$
whose Kobayashi metric is $\omega_{\Delta^n,\mu}= \frac{i}{2}\de\deb  \varphi_{\Delta^n,\mu} $, with:
\begin{equation}\label{eqpolipot}\begin{split} 
  \varphi_{\Delta^n,\mu} (z,w)= -\log\left( \prod_{j=1}^n(1-|z_j|^2)^\mu-|w|^2\right).
\end{split}\end{equation} 
\end{example}

\subsection{Holomorphic isometries between Cartan--Hartogs domains}\label{sectot}
A totally geodesic complex immersion $f:\left(\W',\omega^\prime_{hyp}\right) \f \left(\W,\omega_{hyp}\right)$ between two HSSNCT equipped with their hyperbolic metrics, preserve the triple products $\left\{,,\right\}'$ and $\left\{,,\right\}$ of the associated HPJTS $V'$ and $V$ (see e.g. \cite[Proposition 2.1]{loiscala}), i.e.:
\begin{equation}\label{eqtrpr}
f\{u, v, w\}'=\{f u, f v, f w\}.
\end{equation}
Hence, also the generic norm is preserved, that is, $N_{\Omega}^{\mu}(f(z), \ov  {f(z)})=N_{\Omega'}^{\mu}(z, \ov {z})$.  Thus, the natural lift
\begin{equation}\label{eqliftf}\begin{split} 
\wi{f}: M_{\Omega',\mu} \rightarrow M_{\W,\mu},\qquad \wi{f}(z, w)=(f(z), w),
\end{split}\end{equation}
 is a holomorphic isometric embedding with respect to the Kobayashi metrics defined by \eqref{kobo}, i.e.:
\begin{equation*}\begin{split} 
 \wi f ^* \omega_{\W,\mu}& =- \frac{i}{2} \de\deb \log \left(N_{\Omega}^{\mu}(f(z), \ov  {f(z)})-|w|^{2}\right) \\
 & =- \frac{i}{2} \de\deb \log \left(N_{\Omega'}^{\mu}(z, \ov {z})-|w|^{2}\right) = \omega_{\Omega',\mu}.
\end{split}\end{equation*}
Thus we get:
\begin{prop}\label{propext}
Let $\W$, $\Omega'$ be HSSNCT. Then any totally geodesic complex immersion $f\!:\W' \f \W$ extends to a \K\ embedding $\wi{f}: M_{\Omega',\mu} \rightarrow M_{\W,\mu}$ to the corresponding Cartan--Hartogs domains, defined by \eqref{eqliftf}.
\end{prop}

Consider now the isotropy group $K\subset \operatorname{Aut}(\Omega)$ of the automorphism's group of $\W$ and recall that $K=\Aut\left(V,\left\{,,\right\}\right)$. In fact, by \cite[Prop. III.2.7]{Bertram}, the action of $K$ preserves the triple product $\left\{,,\right\}$ of the associated HPJTS, that is $K\subseteq\Aut\left(V,\left\{,,\right\}\right)$, where $\Aut\left(V,\left\{,,\right\}\right)$ is the group of complex linear transformations of $V$ preserving $\left\{,,\right\}$. Vice versa, as a transformation $f\in \Aut\left(V, \left\{,,\right\}\right)$ preserve the triple product, it preserve also the generic norm $N(x,\ov  y)$. Hence 
\begin{equation*}\begin{split} 
 f ^* \omega_{hyp}& =- \frac{i}{2} \de\deb \log \left(N(f(z),\ov  {f(z)})\right) =- \frac{i}{2} \de\deb \log \left(N(z, \ov {z})\right) = \omega_{hyp},
\end{split}\end{equation*}
that is $K\supseteq\Aut\left(V,\left\{,,\right\}\right)$. 
Then, by the argument above, the holomorphic isometric action of $K$ on $\Omega$ induces in a natural way an holomorphic isometric action of $K$ on $M_\W$, by:
\begin{equation}\label{eqautac}
\tau \cdot \left(z,w\right) =  \left(\tau(z),w\right), \quad \tau\in\operatorname{Aut}(\Omega). 
\end{equation}
Moreover  as a consequence of Propositon \ref{propext} and of the Polydisc Theorem for HSSNCT  (see \cite{helgason}), a Cartan--Hartogs domain can be realized as a union of \K\  embedded {Hartogs--Polydiscs} $M_{\Delta^r,\mu}$:
$$
M_{\Omega,\mu} = \cup_{\tau\in K}\, \tau  \left(M_{\Delta^r,\mu}\right),
$$ 
where $r$ is the rank of $\W$ and $\Delta^r\subset \W$ is a $r$-dimensional complex polydisc totally geodesically embedded in $\Omega$.

\subsection{Proof of Theorem \ref{thmmain}}\label{darbouxsection}
Let $M_{\W,\mu}$ be an $(n+1)$-dimensional  Cartan-Hartogs domain and $\left(\C^n , \left\{,,\right\}_\W\right)$ the HJPTS associated to $\W$. Define the map $\Psi_{\W,\mu}\!:M_{\W,\mu} \rightarrow \C^{n+1}$ by:
\begin{equation}\label{eqPsi}
\Psi_{\W,\mu}(z,w)=\frac{1}{\sqrt{N^\mu_\Omega(z,\ov z)-|w|^2}}\left(\sqrt{\mu N^\mu_\Omega(z,\ov z)}\, \operatorname{B}_\W(z, \ov z)^{-\frac{1}{4}}z, w\right),
\end{equation}
where $\operatorname{B}_\W$ and $N_\W$ are respectively the Bergman operator and the generic norm associated to $\left\{,,\right\}_\W$. 

In order to prove Theorem \ref{thmmain} we will show that $\Psi_{\Omega,\mu}$ satisfies the following properties:
\begin{enumerate}[{\rm (A)}]
\item  
$
\Psi^*_{\Omega,\mu} \w_0 = \w_{\W,\mu},
$
where  $ \w_0=\frac{i}{2}\sum_{j=1}^{n+1} dz_j\wedge d\ov z_j$;
\item $\Psi_{\Omega,\mu}$ is a diffeomorphism.
\end{enumerate}
%Consider an HSSNCT $\W$ and the associated HPJTS $\left({\C^n}, \{ ,  ,\}_\W\right)$.
Let us start with the following two lemmata. 
\begin{lemma}\label{darbouxh}
Let $f_\Omega\!:M_{\Omega,\mu}\rightarrow \mathds C^{n+1}$ be a smooth map of the form:
$$
f_\Omega(z_1,\dots, z_n,w):=\frac{1}{\sqrt{N^\mu_\Omega \left(z,\ov z\right)-|w|^2}}(h_1(z),\dots, h_n(z),w)
$$
where $h:=(h_1,\dots, h_n)$ satisfies:
\begin{equation}\label{condition1}
- \partial\bar\partial N_\Omega^\mu=\sum_{j=1}^ndh_j\wedge d\bar h_j,
\end{equation}
 \begin{equation}\label{condition2}
\sum_{j=1}^n(h_jd\bar h_j-\bar h_jdh_j)=(\partial-\bar \partial) N^\mu.
\end{equation}
Then
\begin{equation}\label{eqfpullb}\begin{split} 
\omega_{\W,\mu}=\frac{i}{2}f_\Omega^*\left(\sum_{j=1}^n dz\wedge d\bar z +  dw\wedge\bar d w\right).
\end{split}\end{equation}
\end{lemma}
\begin{proof}

Observe first that:
\begin{equation}\label{omegamu1}
\begin{split}
\omega_{\W,\mu}=&-\frac i{2}\partial\bar\partial\log\left(N^\mu-|w|^2\right)\\
=&-\frac i{2}\partial\frac{\bar\partial (N^\mu-|w|^2)}{N^\mu-|w|^2}\\
=&\frac i{2}\left[-\frac{ \partial\bar\partial (N^\mu-|w|^2)}{N^\mu-|w|^2}+\frac{ \partial(N^\mu-|w|^2)\wedge  \bar\partial(N^\mu-|w|^2)}{(N^\mu-|w|^2)^2}\right],\\
=&\frac i{2}\left[\frac{dw\wedge d\bar w}{N^\mu-|w|^2}-\frac{ \partial\bar\partial N^\mu}{N^\mu-|w|^2}+\frac{ (\partial N^\mu-\bar wdw)\wedge  \bar\partial(N^\mu-|w|^2)}{2(N^\mu-|w|^2)^2}\right.\\
&\left.+\frac{ \partial(N^\mu-|w|^2)\wedge ( \bar\partial N^\mu-wd\bar w)}{2(N^\mu-|w|^2)^2}\right]\\
\end{split}
\end{equation}
Now, let us compute $\frac{i}{2}f_\Omega^*\left(\sum_{j=1}^n dz\wedge d\bar z +  dw\wedge d\bar  w\right)$. To simplify the indices notation let us write $h_0=z_0=w$. We have:
\begin{equation}\label{psistar1}
\begin{split}
\frac{i}{2}\sum_{j=0}^n d(f_\Omega)_j\wedge d\overline{(f_\Omega)}_j=&\frac{i}{2}\sum_{j=0}^n d\left[\frac{h_j}{\sqrt{ N^\mu-|w|^2}}\right]\wedge d\left[\frac{\bar h_j}{\sqrt{ N^\mu-|w|^2}}\right]\\
=&\frac{i}{2}\sum_{j=0}^n \left[\frac{dh_j\wedge d \bar h_j}{N^\mu-|w|^2}+\frac{h_jd\bar h_j-\bar h_jdh_j}{2(N^\mu-|w|^2)^2}\wedge d(N^\mu-|w|^2)\right]\\
=&\frac{i}{2} \left[\frac{dw\wedge d \bar w}{N^\mu-|w|^2}+\frac{\sum_{j=1}^n dh_j\wedge d \bar h_j}{N^\mu-|w|^2}+\right.\\
+&\left.\frac{wd\bar w-\bar wdw}{2(N^\mu-|w|^2)^2}\wedge d(N^\mu-|w|^2)+\frac{\sum_{j=1}^n\left(h_jd\bar h_j-\bar h_jdh_j\right)}{2(N^\mu-|w|^2)^2}\wedge d(N^\mu-|w|^2)\right].
\end{split}
\end{equation}

Assume that \eqref{condition1} and thus \eqref{condition2} hold. Then from \eqref{omegamu1} and \eqref{psistar1} we get:
\begin{equation}\label{compare}
\begin{split}
f_\Omega^*\omega_0-\omega_\mu=&\frac{i}{{4(N^\mu-|w|^2)^2}}\left[(wd\bar w-\bar wdw)\wedge d(N^\mu-|w|^2)+{(\partial-\bar \partial)N^\mu}\wedge d(N^\mu-|w|^2)\right.\\
&\left.-{ (\partial N^\mu-\bar wdw)\wedge  \bar\partial(N^\mu-|w|^2)}-{ \partial(N^\mu-|w|^2)\wedge ( \bar\partial N^\mu-wd\bar w)}\right]\\
=&\frac{i}{{4(N^\mu-|w|^2)^2}}\left[-\bar\partial(N^\mu -| w|^2)\wedge \bar \partial (N^\mu-|w|^2)+\partial(N^\mu-|w|^2)\wedge \partial(N^\mu-|w|^2)\right]\\
=&0,
\end{split}
\end{equation}
and we are done.
\end{proof}

\begin{lemma}\label{raffaello}
Let $F\!:(\Omega, \omega_{hyp}) \rightarrow (\mathds C^n, w_0)$ be a holomorphic map satisfying $F^*\omega_0=\omega_{hyp}$, and:
\begin{equation}\label{la19}
\sum_{j=1}^n\left( F_jd\bar F_j-\bar F_jdF_j\right)=\partial \log N-\bar \partial \log N.
\end{equation}
 Then:
\begin{equation}\label{compare}
- \partial\bar\partial N^\mu=\mu\sum_{j=1}^dd(N^{\mu/2}F_j)\wedge d (N^{\mu/2}\bar F_j),
\end{equation}
and
 \begin{equation}\label{condition2proven}
\mu N^{\mu/2}\sum_{j=1}^d\left(F_j d (N^{\mu/2}\bar F_j)-\bar F_j d (N^{\mu/2} F_j)\right)=(\partial-\bar \partial) N^\mu.
\end{equation}
\end{lemma}
\begin{proof}
We start by proving \eqref{compare}. Observe first that from $F^*\omega_0=-\frac i2\partial\bar \partial \log N$ one gets that $\sqrt{\mu}F$ satisfies $\sqrt{\mu}F^*\omega_0=-\frac i2 \partial \bar \partial \log N^\mu$, as it follows by:
$$
\mu\sum_{j=1}^ddF_j\wedge d\bar F_j=- \partial \bar \partial \log N^\mu.
$$
Then, expanding the right hand side of \eqref{compare} we get:
\begin{equation}
\begin{split}
\mu\sum_{j=1}^dd(N^{\mu/2}F_j)\wedge d (N^{\mu/2}\bar F_j)=&\mu\sum_{j=1}^d(F_jdN^{\mu/2}+N^{\mu/2}dF_j)\wedge (\bar F_jd N^{\mu/2}+N^{\mu/2}d\bar F_j)\\
=&\mu\sum_{j=1}^d\left[F_jN^{\mu/2}dN^{\mu/2}\wedge d\bar F_j+ \bar F_jN^{\mu/2}dF_j\wedge d N^{\mu/2}+N^{\mu}dF_j\wedge d\bar F_j\right]\\
=&\mu N^{\mu/2}\sum_{j=1}^d\left[dN^{\mu/2}\wedge F_jd\bar F_j+\bar F_jdF_j\wedge d N^{\mu/2}\right]-N^\mu\partial\bar \partial \log N^\mu\\
=&\mu N^{\mu/2}\sum_{j=1}^d\left[dN^{\mu/2}\wedge F_jd\bar F_j+\bar F_jdF_j\wedge d N^{\mu/2}\right]+\\
&-\frac{N^\mu\partial\bar \partial  N^\mu-\partial N^\mu\wedge \bar \partial N^\mu}{N^{\mu}}.
\end{split}\nonumber
\end{equation}
At this point, \eqref{compare} is satisfied if and only if:
\begin{equation}
\begin{split}
\mu N^{\mu/2}\sum_{j=1}^d\left[dN^{\mu/2}\wedge F_jd\bar F_j+\bar F_jdF_j\wedge d N^{\mu/2}\right]=&-4\partial N^{\mu/2}\wedge \bar \partial N^{\mu/2}\\
=&-2d N^{\mu/2}\wedge \bar \partial N^{\mu/2}-2\partial N^{\mu/2}\wedge d N^{\mu/2}
\end{split}\nonumber
\end{equation}
where we used that $ \partial N^\mu=2 N^{\mu/2}\partial N^{\mu/2}$.
This last equivalence can be rewritten as:
$$
\sum_{j=1}^d\left[dN^{\mu/2}\wedge F_jd\bar F_j+\bar F_jdF_j\wedge d N^{\mu/2}\right]=-d N^{\mu/2}\wedge \bar \partial \log N-\partial \log N\wedge d N^{\mu/2}
$$
that is:
$$
dN^{\mu/2}\wedge\sum_{j=1}^d\left[ F_jd\bar F_j-\bar F_jdF_j\right]=-d N^{\mu/2}\wedge \left[\bar \partial \log N-\partial \log N\right],
$$
which holds true once \eqref{la19} does.

Finally, the following computation proves \eqref{condition2proven}:
\begin{equation}
\begin{split}
\mu N^{\mu/2}&\sum_{j=1}^d\left(F_j d (N^{\mu/2}\bar F_j)-\bar F_j d (N^{\mu/2} F_j)\right)=\mu N^{\mu/2}\sum_{j=1}^d\left(F_j (N^{\mu/2}d\bar F_j+\bar F_j dN^{\mu/2})+\right.\\
&\qquad\qquad\qquad\qquad\qquad\qquad\qquad\qquad\qquad\qquad\left.-\bar F_j (N^{\mu/2} dF_j+F_jdN^{\mu/2})\right)\\
=&\mu N^{\mu}\sum_{j=1}^d\left(F_jd\bar F_j-\bar F_jdF_j\right)=\mu N^\mu\left(\partial \log N-\bar \partial \log N\right)\\
=&\mu N^{\mu-1}\left(\partial  N-\bar \partial  N\right)=\partial  N^\mu-\bar \partial  N^\mu.
\end{split}\nonumber
\end{equation}
\end{proof}

%where  $ \w_0=\frac{i}{2}\sum_{j=1}^{n+1} dz_j\wedge d\ov z_j$;

Now we can proceed with the proof of (A).
In \cite[Theorem 1.1]{loiscala}, A. Loi and A. Di Scala show that the holomorphic map $F\!:(\Omega, \omega_{hyp}) \rightarrow (\mathds C^n, w_0)$ defined by:
$$
F(z)=\operatorname{B}_\Omega(z, \ov z)^{-\frac{1}{4}} z, 
$$
is a global symplectomorphism. Thus, by Lemma \ref{darbouxh} and Lemma \ref{raffaello}, in order to prove (A) we need only to check that such $F$ satisfies \eqref{la19}.

Denote by $\operatorname{D}(x,y)$ the operator on $(V,\{,,\})$ defined by $\operatorname D(x,y)z=\{x,y,z\}$
and denote by $z=\sum_{j}\lmb_j c_j$ the spectral decomposition of $z$. We have (see \cite[(28)]{loiscala}) that:
\begin{equation*}\begin{split} 
 F(z)=\left(\mathrm{id}-\frac{1}{2} \operatorname{D}(z, z)\right)^{-1 / 2} z=\left(\mathrm{id}-z \square z\right)^{-1 / 2}\, z
\end{split}\end{equation*}
where we use the operator $z \square z:=\frac{1}{2} \operatorname{D}(z, z)$. Therefore,
\begin{equation*}\begin{split} 
&\qquad \sum_{j=1}^d\left[ F_jd\ov F_j-\ov F_jdF_j\right] = \, m_1\left(F,dF\right) - \, m_1\left(dF,F\right)\\
&=   m_1 \left(\left(\mathrm{id}-z \square z\right)^{-1 / 2}\, z
, d\left(\left(\mathrm{id}-z \square z\right)^{-1 / 2}\, z
\right)\right)
-  m_1 \left(d\left(\left(\mathrm{id}-z \square z\right)^{-1 / 2}\, z
\right),\left( \mathrm{id}-z \square z\right)^{-1 / 2}\, z\right)\\
&=    m_1 \left(\left(\mathrm{id}-z \square z\right)^{-1 / 2}\, z
, \left(\mathrm{id}-z \square z\right)^{-1 / 2}\, dz
\right)
-  m_1 \left(\left(\mathrm{id}-z \square z\right)^{-1 / 2}\, dz
,\left( \mathrm{id}-z \square z\right)^{-1 / 2}\, z\right),
\end{split}\end{equation*}
where we used the identity:
\[
\mathrm{d} F(z)=\left(\mathrm{d}(\mathrm{id}-z \square z)^{-1 / 2}\right) z+(\mathrm{id}-z \square z)^{-1 / 2}\mathrm{d} z,
\]
and the fact that $z \square z$ is self-adjoint with respect to the Hermitian metric $m_{1}$. Using \cite[(34)]{loiscala} we get:
$$
m_{1}\left(F(z),\left((\mathrm{id}-z \square z)^{-1 / 2}\right) \mathrm{d} z\right)=-\frac{\ov{\partial} {N(z, \ov  z)}}{{N(z, \ov  z)}},
$$
and thus:
\begin{equation*}\begin{split} 
 \de \log N(z, \ov  z)&- \deb \log N(z,\ov z)=\frac{{\partial} {N(z,\ov z)}}{{N(z,\ov z)}}-\frac{{\ov\partial} {N(z,\ov z)}}{{N(z,\ov z)}}\\
 & = m_{1}\left(F(z),\left((\mathrm{id}-z \square z)^{-1 / 2}\right) \mathrm{d} z\right) - m_{1}\left(\left((\mathrm{id}-z \square z)^{-1 / 2}\right) \mathrm{d} z ,F(z)\right),
\end{split}\end{equation*}
as wished.

In order to prove (B), %\label{secb}%\label{secgcas}
we need the following lemma:
\begin{lemma}\label{secHP}
Property (B) holds for Hartogs--Polydiscs.
\end{lemma}
\begin{proof}
Let $M_{\Delta^n,\mu}$ be an $n$-dimensional Hartogs--polydisc, as described in Example \ref{hartogspoly}. 
Recall that (see e.g. \cite[Sec. 3]{loiscala}) the HPJTS associated to the polydisc $\Delta^n$ is $\C^n$ equipped with the triple product $\left\{z,u,w\right\}=j^{-1}\left(j(z)j(u)^*j(w)+j(w)j(u)^*j(z)\right)$, its Bergman operator is $\operatorname{B}(u, v) w=j^{-1}\left(\left(I_{n}-j(u) j(v)^{*}\right) j(w)\left(I_{n}-j(v)^{*} j(u)\right)\right)$ and the flat \K\ form is given by:
$$
\w_0=\frac{i}{2} \de\deb \tr j(z)j(z)^* = \frac{i}{2} \de\deb |z|^2.
$$
Thus \eqref{eqPsi} reads:
{\small\begin{equation}\label{eqPsi0} 
\Psi_{\Delta^n}(z,w)=\frac{1}{\sqrt{  \prod_{j=1}^n(1-|z_j|^2)^\mu-|w|^2}}\left(\sqrt{\mu\prod_{j=1}^n(1-|z_j|^2)^{\mu}} \left(\frac{z_1}{\sqrt{1-|z_1|^2}},\dots,\frac{z_n}{\sqrt{1-|z_2|^2}}\right),{ w}\right).
\end{equation}}
Let:
$$
\wi{M}=\left\{\left(x_{1}, \ldots, x_{n},y\right) \in \mathbb{R}^{n+1}\midd x_{j}=\left| z_{j}\right|^{2}, \, y=\left| w \right|^{2},\, z=\left(z_{1}, \ldots z_{n},w\right) \in M_{\Delta^n}\right\},%\subset \R^{n+1}
$$
and consider the smooth map $\wi\varphi_{\Delta^n,\mu}:\wi M \f \R$, defined by:
\begin{equation}\begin{split} %\label{eqphit}
\tilde\varphi_{\Delta^n,\mu}(x_1,\dots, x_n,y):=\varphi_{\Delta^n,\mu}(|z_1|^2,\dots, |z_n|^2,|w|^2)=-\log \left(\prod_{j=1}^{n}\left(1-x_j\right)^{\mu}-y\right).
\end{split}\nonumber\end{equation}
By \cite[Th. 1.1]{loizuddas}, $\Psi_{\Delta^n,\mu}(z,w)$ is a global symplectomorphism (in particular a diffeomorphism) if $\frac{\partial \wi\varphi_{\Delta^n,\mu}}{\partial x_{k}}>0$, $\frac{\partial \wi\varphi_{\Delta^n,\mu}}{\partial y}>0$ and $\lim _{(x,y) \rightarrow \partial M} \sum_{j=1}^{n} \frac{\partial \wi\varphi_{\Delta^n,\mu}}{\partial x_{j}} x_{j}+\frac{\partial \wi\varphi_{\Delta^n,\mu}}{\partial y}y=+\infty$. The first two conditions are easily checked:
$$
\frac{\partial \wi \varphi_{\Delta^n,\mu}}{\partial x_j} =\frac{\mu\prod_{j=1}^n(1-x_j)^{\mu}}{(1-x_j)\left( \prod_{j=1}^n(1-x_j)^\mu-y\right)}>0,
$$
$$
\frac{\partial \tilde\varphi_{\Delta^n,\mu}}{\partial y} =\frac{1}{ \prod_{j=1}^n(1-x_j)^\mu-y}>0.
$$
It remains to verify the third condition:
\begin{equation}\label{globalcond}
\lim_{(x,y)\rightarrow \partial M_{\Delta^n,\mu}}\left( \frac{\mu\prod_{k=1}^n(1-x_k)^{\mu}}{\left( \prod_{k=1}^n(1-x_k)^\mu-y\right)}\sum_{j=1}^n\frac{x_j}{1-x_j}+\frac{y}{ \prod_{j=1}^n(1-x_j)^\mu-y}\right)=+\infty.
\end{equation}
Observe that $\partial M_{\Delta^n,\mu}=\partial \Delta^n\cup \{y=\prod_{k=1}^n(1-x_k)^{\mu}\} $, thus \eqref{globalcond} is satisfied since for any $l=1,\dots, n$:
$$
\lim_{x_l\rightarrow 1}\left( \frac{\mu\prod_{k=1}^n(1-x_k)^{\mu}}{\left( \prod_{k=1}^n(1-x_k)^\mu-y\right)}\sum_{j=1}^n\frac{x_j}{1-x_j}+\frac{y}{ \prod_{j=1}^n(1-x_j)^\mu-y}\right)=+\infty,
$$
and 
$$
\lim_{y\rightarrow \prod_{k=1}^n(1-x_k)^{\mu}}\left( \frac{\mu\prod_{k=1}^n(1-x_k)^{\mu}}{\left( \prod_{k=1}^n(1-x_k)^\mu-y\right)}\sum_{j=1}^n\frac{x_j}{1-x_j}+\frac{y}{ \prod_{j=1}^n(1-x_j)^\mu-y}\right)=+\infty,
$$
concluding the proof.
\end{proof}
We can proceed with the proof of (B). Let us consider a Cartan domain $\W$ of rank $r$ and let 
$$
z=\lambda_{1} c_{1}+\lambda_{2} c_{2}+\cdots+\lambda_{r} c_{r},
$$
be a spectral decomposition of $z \in \W$. Since:
$$
\operatorname{B}_\Omega(z, \ov z) c_{j}=\left(1-\lambda_{j}^{2}\right)^{2} c_{j} 
$$
and 
$$
N_\Omega(z,\ov z)=\prod_{j=1}^r(1-\lmb^2),
$$
(see \cite{roos}, and also \cite{loiscala}) we can write $\Psi_{\W,\mu}$ as follows:
\begin{equation}\begin{split} \label{eqPsij}
 \Psi_{\Omega,\mu}(z,w)=\frac{1}{\sqrt{\prod_{j=1}^r \left(1-\lmb^2\right)^\mu-|w|^2}}\left(\sqrt{\mu \prod_{j=1}^r \left(1-\lmb^2\right)^\mu}\sum_{j=1}^{r} \frac{\lambda_{j}}{\left(1-\lambda_{j}^{2}\right)^{1 / 2}} c_{j}, w\right).
\end{split}\end{equation}
Comparing  \eqref{eqPsij} with \eqref{eqPsi0} and using Lemma \ref{secHP}, we deduce that $ \Psi_{\Omega,\mu}$ is a diffeomorphism (we apply \cite[Section 1.6]{dilr}).
%{\color{red} 
%\begin{remark}\rm\label{rmkpsir1m1}
%For $r=1$ and $\mu=1$, $M_{\W,\mu}=\C H ^{n+1}$ and from 
%$$
%\Psi_{\Omega,\mu} (z,w)= \frac{1}{\sqrt{1-|z|^2-|w|^2}}(z,w) 
%$$
%Namely the symplectic duality the hyperbolic space and its dual, defined by A. Di Scala and A. Loi in \cite{loiscala}.
%\end{remark}
%}

\medskip
%\subsection{Concluding remarks}
%We conclude this section with the following remarks that highlight some properties of the map $\Psi_{\Omega,\mu}$, which are the same holding for the corresponding map for hermitian symmetric spaces of noncompact type, proven in \cite{loiscala}.

\begin{remark}[Hereditary property]\label{heredrem}\rm
Observe that the map $\Psi_{\Omega,\mu}$ is hereditary in the following sense:
for any bounded symmetric domain $\Omega'\subset \C^m$  complex and totally geodesic embedded  $\Omega' \stackrel{f}{\hookrightarrow} \Omega$, such that $f(0)=0$, one has:
 $$
 \Psi_{\Omega',\mu}(z,w)={\Psi_{\Omega,\mu}}(f(z),w).
 $$
Indeed, consider a complex and totally geodesic embedded submanifold $f\!:\Omega' {\hookrightarrow} \Omega$, satisfying $f(0)=0$. By Prop. \ref{propext}, $f$ lifts to a \K\ embedding  $\tilde f\!:M_{\Omega'} {\hookrightarrow} M_{\Omega}$, defined by $\wi f (z,w)=(f(z),w)$. By \cite[Prop. 2.2]{loiscala}, $f$ preserves the triple products, thus it follows that it preserves also the Bergman operator $\operatorname{B}_\Omega$ and the generic norm $N_\Omega$. Hence $\Psi_{\Omega^{\prime},\mu}={{\Psi_{\Omega,\mu}}}(f(z),w)$.
\end{remark}

\begin{remark}\label{autrem}\rm
The map $\Psi_{\Omega,\mu}$ commutes with the holomorphic and isometric action \eqref{eqautac} of the isotropy group $K \subset \operatorname{Aut}(\Omega)$ at the origin, i.e. for every $\tau \in K$, $\Psi_{\Omega,\mu} \circ \tau=\tau \circ \Psi_{\Omega,\mu}$.
This follows since $K = \operatorname{Aut}(\C^n, \left\{,,\right\}_\W)$, and therefore: 
\begin{equation*}\begin{split} 
 \Psi_{\Omega,\mu} \circ \tau \left(z,w\right)&=\frac{1}{\sqrt{N^\mu_\Omega(\tau(z),\ov {\tau(z)})-|w|^2}}\left(\sqrt{\mu N^\mu_\Omega(\tau(z),\ov{\tau(z)})}\,\operatorname{B}_\Omega(\tau(z),\ov {\tau(z)})^{-\frac{1}{4}}\tau(z), w\right)\\
 &= \frac{1}{\sqrt{N_{\Omega}^{\mu}(z, \bar{z})-|w|^{2}}}\left(\sqrt{\mu N_{\Omega}^{\mu}(z, \bar{z})} \,\operatorname{B}_\Omega(z, \ov z)^{-\frac{1}{4}}\tau( z), w\right)\\
 & = \tau \circ \Psi_{\Omega,\mu}  \left(z,w\right).
\end{split}\end{equation*}
\end{remark}

\begin{remark}[Alterative proof of Theorem \ref{thmmain} for classical Cartan-Hartogs]\rm
It is possible to give a more geometric proof of (A) for Cartan--Hartogs domains based on Cartan domains of classical type, without using the Jordan triple system theory. 
A direct computation proves \eqref{la19} for the Cartan--Hartogs $M_{D_I,\mu}$ based on the first classical domain $D_I$ (see e.g. \cite{hua}). %A HSSNCT of classical type $\W$, the related HPJTS $\left(\C^n,\{,,\}_\W\right)$ and the  associated Cartan--Hartogs domain $M_{\Omega,\mu}$.
 It is known that a HSSNT $\W$ admits a complex and totally geodesic embedding $f$ in $D_{I}[m,m],$ for $m$ sufficiently large. (This is obviously true for the domains $D_{I}, D_{I I}$ and $D_{I I I},$ while for the domain $D_{I V}$ -- associated to the so-called Spin-factor -- the explicit embedding can be found in \cite{crim}.) By Proposition \ref{propext}, $f$ lifts to  a complex and totally geodesic embedding $\wi f$ of $M_\W$ in $M_{D_{I}[n],\mu}$. We can assume that this embedding takes the origin $0 \in M$ to the origin $0 \in D_{I}[n]$. % Therefore the $\left(\C^n,\{,,\}_\W\right)$ is a sub-HPJTS of $\left(\mathbb{C}^{m^{2}},\{,,\}\right)$ (see Prop. \ref{propsub}). 
Hence property (A) for $M_{\W,\mu}$ is a consequence of the Hereditary property
given in Remark \ref{heredrem} and the fact that (A) holds true for $D_{I}[m,m]$.
\end{remark}

\section{Dual Cartan--Hartogs domains}\label{dualsection}
\subsection{Definition and geometric properties} We define the {\em dual Cartan--Hartogs domain} $M^*_{\Omega,\mu}$ as $\C^{n+1}$, equipped with the dual K\"ahler form (see Lemma \ref{lemswkahl}):
\begin{equation}\label{eqdualkf}\begin{split} 
\sw_{\W,\mu}=\frac{i}{2}\de\deb \varphi^*_{\Omega,\mu},  
\end{split}\end{equation}
where $ \varphi^*_{\Omega,\mu}:= \log(N^\mu_{\W}(z,- \ov  z) + |w|^2)$ is the dual K\"ahler potential  %is obtained from \eqref{phicartan} by $ \varphi_{\Omega^*,\mu}(z,w;\bar z,\bar w):=-\ \varphi_{\Omega,\mu}(z,w;-\bar z,-\bar w)$, according to the definition of duality given in \cite{lmbdual}
(see the introduction for the definition of symplectic dual). If we restrict $\omega^* _{\W,\mu}$ to $\C^{n}=\left\{(z,0)\in \C^{n+1}\right\}$ we get a multiple of the \K\ form dual to the hyperbolic form \eqref{eqhypd}, i.e.:
$$
\omega^*_{{\Omega,\mu}_{\mid_{\C^n}}}={\mu}\,\w^*_{hyp}.
$$

\begin{example}[Dual Hartogs--polydisc]\label{hartogspolydual}\rm
Consider the Hartogs--polydisc $M_{\Delta^n,\mu}$ of the previous example. Then, by \eqref{eqpolipot} and \eqref{eqdualkf}, the dual form on $\C ^{n+1}$ is given by $\w_{\Delta^n,\mu}^*=\frac{i}{2} \de\deb \varphi^*_{\Delta^n,\mu}$, where:
$$
 \varphi^*_{\Delta^n,\mu}(z,w) = \log\left( \prod_{j=1}^n(1+|z_j|^2)^\mu+|w|^2\right).
$$
\end{example}

In general the dual of a \K\ form is not defined (see \cite[Example 1.3]{lmbdual}), the following lemma assures us that $\sw_{\W,\mu}$ is a \K\ metric on $\C^{n+1}$.

\begin{lemma}\label{lemswkahl}
The function $\varphi^*_{\Omega,\mu}\!:\mathds C^{n+1}\rightarrow \mathds R$ is strictly plurisubharmonic.
\end{lemma}
\begin{proof}
To shorten the notation let us write $\tilde N$ for $N_\Omega(z,-\ov z)$. 
 The hessian of $\varphi^*_{\Omega,\mu}$ is given by:
$$
H:=\frac{1}{\left(\tilde N^\mu+|w|^2\right)^2}\left(\begin{array}{c|c}
\left(\tilde N^\mu+|w|^2\right)\partial_j\bar\partial_k\tilde N^\mu-\partial_j\tilde N^\mu\bar \partial_k\tilde N^\mu&-w\partial_j\tilde N^\mu\\
\hline
-\bar w\bar\partial_k\tilde N^\mu&\tilde N^\mu
\end{array}\right).
$$
Observe that it is enough to show that $H$ is positive definite when $\Omega=\Delta^n$. Indeed, let $(\C^r, \w^*_{\Delta})$ be the symplectic dual of $(\Delta, \w_{hyp})$, i.e. $\w^*_{\Delta}=\frac{i}{2}\de\deb \log \left(\prod_{j=1}^r\left(1+|z_j|^2\right)\right)$, and let $(z_0,w_0)\in \C ^{n+1}$, $(u,\xi)\in T _{(z_0,w_0)}\C^{n+1}$. By the dual Polydisk Theorem, there exists a totally geodesic holomorphic immersion $f:(\C^r, \w^*_{\Delta}) \f \left(\C^n, \w^*_{hyp}\right)$ such that $f(\wi z_0)=z_0$ e $f_{*,\wi z_0}(\wi u)=u$, for suitable $\wi z_0 \in \C ^n$ and $\wi u \in T_{\wi z_0} \C^r$. Then by \eqref{eqpullbw} below:
$$
g_{\Delta,\mu}^*\left(\left(\wi u, \xi\right),\left(\wi u, \xi\right)\right)=g_{\W,\mu}^*\left(\wi f _{*,\left(z_0,w_0\right)}\left(\wi u, \xi\right),\wi f _{*,\left(z_0,w_0\right)}\left(\wi u, \xi\right)\right)=g_{\W,\mu}^*\left(\left( u, \xi\right),\left( u, \xi\right)\right)
$$
where $\wi f \left(z,w\right)=\left(f \left(z\right),w\right)$. 
Thus, consider a dual Hartogs--polydisk of dimension $n+1$. Then $\tilde N^\mu=\prod_{h=1}^n(1+|z_h|^2)^\mu$ and thus for $j, k=1,\dots, n$:
\begin{equation}
\begin{split}
\bar\partial_k\prod_{h=1}^n(1+|z_h|^2)^\mu%=&\mu(1+|z_k|^2)^{\mu-1} z_k\prod_{k\neq h=1}^n(1+|z_h|^2)^\mu\\
=&\frac{\mu\, z_k\,\prod_{ h=1}^n(1+|z_h|^2)^\mu}{1+|z_k|^2},
\end{split}\nonumber
\end{equation}
%
%\begin{equation}
%\begin{split}
%\partial_k\bar\partial_k\prod_{h=1}^n(1+|z_h|^2)^\mu=&\mu(\mu-1)(1+|z_k|^2)^{\mu-2} |z_k|^2\prod_{k\neq h=1}^n(1+|z_h|^2)^\mu\\
%&+\mu(1+|z_k|^2)^{\mu-1}\prod_{k\neq h=1}^n(1+|z_h|^2)^\mu\\
%=&\frac{\mu\,\prod_{ h=1}^n(1+|z_h|^2)^\mu}{(1+|z_k|^2)^{2}}\left((\mu-1)|z_k|^2+1+|z_k|^2\right)\\
%=&\frac{\mu\,\prod_{ h=1}^n(1+|z_h|^2)^\mu}{(1+|z_k|^2)^{{2}}}\left(1+\mu|z_k|^2\right)
%\end{split}\nonumber
%\end{equation}
%Let $j\neq k$.
%\begin{equation}
%\begin{split}
%\partial_j\bar\partial_k\prod_{h=1}^n(1+|z_h|^2)^\mu=&\mu^2(1+|z_k|^2)^{\mu-1}(1+|z_j|^2)^{\mu-1} \bar z_jz_k\prod_{k,j\neq h=1}^n(1+|z_h|^2)^\mu\\
%=&\frac{\mu^2\, z_k\bar z_j\,\prod_{ h=1}^n(1+|z_h|^2)^\mu}{(1+|z_k|^2)(1+|z_j|^2)}
%\end{split}\nonumber
%\end{equation}
%Quindi possiamo direttamente scriverle con $j$, $k$ generici:
$$
\partial_j\bar\partial_k\prod_{h=1}^n(1+|z_h|^2)^\mu=\frac{\mu\,\prod_{ h=1}^n(1+|z_h|^2)^\mu}{(1+|z_k|^2)(1+|z_j|^2)}\left({ {\delta_{jk}}}+\mu\,z_k\bar z_j\right).
$$
Thus:
\begin{equation}
\begin{split}
\left(\tilde N^\mu+|w|^2\right)\partial_j\bar\partial_k\tilde N^\mu-\partial_j\tilde N^\mu\bar \partial_k\tilde N^\mu%=&\left(\prod_{ h=1}^n(1+|z_h|^2)^\mu+|w|^2\right)\frac{\mu\,\prod_{ h=1}^n(1+|z_h|^2)^\mu}{(1+|z_k|^2)(1+|z_j|^2)}\left(\delta_{jk}+\mu\,z_k\bar z_j\right)\\
%&-\frac{\mu^2\, z_k\bar z_j\,\prod_{ h=1}^n(1+|z_h|^2)^{2\mu}}{(1+|z_k|^2)(1+|z_j|^2)} {\color{red} }\\
=&\left(\prod_{ h=1}^n(1+|z_h|^2)^\mu+|w|^2\right)\frac{\mu\,\prod_{ h=1}^n(1+|z_h|^2)^\mu}{(1+|z_k|^2)(1+|z_j|^2)}\delta_{jk}+\\
&+\frac{\mu^2\,z_k\bar z_j\,|w|^2\prod_{ h=1}^n(1+|z_h|^2)^\mu}{(1+|z_k|^2)(1+|z_j|^2)}
\end{split}\nonumber
\end{equation}
Setting:
$$
A:=\left(\prod_{ h=1}^n(1+|z_h|^2)^\mu+|w|^2\right)\mu\,\prod_{ h=1}^n(1+|z_h|^2)^\mu\left(\begin{array}{ccc}\frac{1}{(1+|z_1|^2)^2}&&\\
&\ddots&\\
&&\frac{1}{(1+|z_n|^2)^2}\end{array}\right)
$$
and:
$$
B:=\mu^2\, \prod_{ h=1}^n(1+|z_h|^2)^\mu VV^*,
$$
where $V$ is the column vector with $k$-th entry $\frac{z_k}{1+|z_k|^2}$, the Hessian $H$ reads:
$$
H:=\frac{1}{\left( \prod_{ h=1}^n(1+|z_h|^2)^\mu+|w|^2\right)^2}\left(\begin{array}{c|c}
A+ B&-w\frac{\mu\, \bar z_j\,\prod_{ h=1}^n(1+|z_h|^2)^\mu}{1+|z_j|^2}\\
\hline
-\bar w\frac{\mu\, z_k\,\prod_{ h=1}^n(1+|z_h|^2)^\mu}{1+|z_k|^2}&\prod_{ h=1}^n(1+|z_h|^2)^\mu
\end{array}\right),
$$
and since $A+B$ is positive definite (being the sum of a positive definite matrix $A$ and a semipositive one $B$), $H$ is positive definite iff its determinant is. A long but straightforward computation gives:
$$
\det(H)=\mu^n\frac{\prod_{h=1}^n(1+|z_h|^2)^{\mu(n+1)-2}}{\left(\prod_{h=1}^n(1+|z_h|^2)^\mu+|w|^2\right)^{n+2}},
$$
and we are done.
\end{proof}

\begin{remark}\rm
Observe that it turns out (see \cite[Subsec. 2.4]{loiscala}) that the Hermitian  symmetric space of compact type $\left(\SW,\w_{FS}\right)$  dual to $\left(\W,\w_{hyp}\right)$ is a compactification of $\left(\C^n, \w^*_{hyp}\right)$. Further, $\left\{\left(z,w \right)\in M^* _{\W,\mu}  \mid  z=0\right\}$ is totally geodesic in $ M^* _{\W,\mu} $ and has $\C{\rm P}^1$ equipped with the Fubini--Study metric as compactification, therefore $M^* _{\W,\mu}$ is not complete for any $\mu$. 
The authors believe that $\left(\C^{n+1}, \sw_{\W,\mu} \right)$ admits a completion only when $M_{\Omega,\mu}$ is itself a Hermitian  symmetric space of noncompact type, which actually happens only when it reduces to be the $(n+1)$-dimensional complex hyperbolic space, i.e. when $\mu=1$ and $\rank(\W)=1$. 
\end{remark}

Using \eqref{francesca} and the following result by A. Selberg \cite{selberg}:
\begin{equation}\label{selberg}
\begin{split}
F(s)=&\mathop{\int\dots \int}_{1>\lmb_1 >\dots>\lmb_r>0} \ \prod_{j=1}^r\left(1-\lmb_j^2\right)^s \prod_{j=1}^r \lmb_j^{2b+1} \prod_{1\leq j<k\leq r} (\lmb_j^2 - \lmb_k^2)^a  \ d \lmb_1 \wedge \dots \wedge d \lmb_r \\
=&\frac1{2^r}\mathop{\int\dots \int}_{1>\lmb_1 >\dots>\lmb_r>0} \ \prod_{j=1}^r\left(1-t_j\right)^s \prod_{j=1}^r t_j^{b} \prod_{1\leq j<k\leq r} (t_j - t_k)^a  \ dt_1 \wedge \dots \wedge d t_r.
\end{split}
\end{equation}
we have the following:
\begin{lemma}\label{volume}
The volume of a $(n+1)$-dimensional dual Cartan--Hartogs domain $\left(\C^{n+1},\w^*_{\W,\mu}\right)$ is given by:
$$
\operatorname {Vol}\left(\C^{n+1},\w^*_{\W,\mu}\right)
=\frac{\pi^{n+1}\mu^n}{(n+1)!}F(0) \int_\mathcal{F} \Theta,
$$
where $\Theta$ is the induced volume form on F\"urstenberg-Satake boundary $\mathcal{F}$ of $\W$.
\end{lemma}
\begin{proof}
Observe first that since (see e.g. \cite{koranyi2}):
$
 \det(\mu\omega^*_{hyp})=\mu^n(N^*)^{-\gamma},
$
a long but straightforward computation gives:
$$
\det(\omega^*_{\W,\mu})=\mu^n\frac{\left(N^*\right)^{\mu(n+1)-\gamma}}{\left(\left(N^*\right)^\mu+|w|^2\right)^{n+2}}.
$$

Thus, using the polar coordinates of HSSNT (see e.g. ), we get:
\begin{equation}
\begin{split}
\operatorname {Vol}&(\C ^{n+1},\omega^*_{\W,\mu})=\int_{\mathds C^{n+1}}\frac{(\omega^*_{\W,\mu})^{n+1}}{(n+1)!}
=\mu^n\pi\int_{\mathds C^n} \int_{0}^{+\infty}\frac{\left(N^*\right)^{\mu(n+1)-\gamma}}{\left(\left(N^*\right)^\mu+r_w\right)^{n+2}}dr_w\wedge\frac{\omega_0^n}{n!}\\
=&\frac{\pi\mu^n}{(n+1)!}\int_{\mathds C^n} \frac{\left(N^*\right)^{\mu(n+1)-\gamma}}{\left(N^*\right)^{\mu(n+1)}}\w_0^n
=\frac{\pi\mu^n}{(n+1)!}\int_{\mathds C^n}\left(N^*\right)^{-\gamma}{\w_0^n}\\
=&\frac{\pi^{n+1}\mu^n}{(n+1)!} \int_\mathcal{F} \Theta\mathop{\int\cdots\int}_{+\infty>\lmb_1 >\dots>\lmb_r>0}\ \prod_{j=1}^r\left(1+\lmb_j^2\right)^{-\gamma} \prod_{j=1}^r \lmb_j^{2b+1} \prod_{1\leq j<k\leq r} (\lmb_j^2 - \lmb_k^2)^a  \ d \lmb_1 \wedge \dots \wedge d \lmb_r\\
=& \frac{\pi^{n+1}\mu^n}{(n+1)!}F(0) \int_\mathcal{F} \Theta,
\end{split}
\end{equation}
where $\gamma=b+2+(r-1)\frac a2$ is the genus of $\Omega$, $F(s)$ is given in \eqref{francesca}, and last equality follows by:
\begin{equation}
\begin{split}
&\mathop{\int\cdots\int}_{+\infty>\lmb_1 >\dots>\lmb_r>0}\ \prod_{j=1}^r\left(1+\lmb_j^2\right)^{-\gamma} \prod_{j=1}^r \lmb_j^{2b+1} \prod_{1\leq j<k\leq r} (\lmb_j^2 - \lmb_k^2)^a  \ d \lmb_1 \wedge \dots \wedge d \lmb_r\\
=&\frac{1}{2^r}\mathop{\int\cdots\int}_{+\infty>\lmb_1 >\dots>\lmb_r>0}\ \prod_{j=1}^r\left(1+t_j\right)^{-\gamma} \prod_{j=1}^r t_j^{b} \prod_{1\leq j<k\leq r} (t_j - t_k)^a  \ d t_1 \wedge \dots \wedge d t_r\\
=&\frac{1}{2^r}\mathop{\int\cdots\int}_{1>\lmb_1 >\dots>\lmb_r>0}\ \prod_{j=1}^r\left(1-s_j\right)^{\gamma-2-b} \prod_{j=1}^r s_j^{b} \prod_{1\leq j<k\leq r} \left( \frac{s_j}{1-s_j}- \frac{s_k}{1-s_k}\right)^a \ d s_1 \wedge \dots \wedge d s_r\\
%=&\frac{1}{2^r}\mathop{\int\cdots\int}_{1>\lmb_1 >\dots>\lmb_r>0}\ \prod_{j=1}^r\left(1-s_j\right)^{\gamma-2-b} \prod_{j=1}^r s_j^{b} \prod_{1\leq j<k\leq r} \left( \frac{s_j-s_k}{(1-s_j)(1-s_k)}\right)^a \ d s_1 \wedge \dots \wedge d s_r\\
=&\frac{1}{2^n}\mathop{\int\cdots\int}_{1>\lmb_1 >\dots>\lmb_r>0}\ \prod_{j=1}^r\left(1-s_j\right)^{\gamma-2-b-(r-1)a} \prod_{j=1}^r s_j^{b} \prod_{1\leq j<k\leq r} ( s_j-s_k)^a \ d s_1 \wedge \dots \wedge d s_r\\
=&\frac{1}{2^r}\mathop{\int\cdots\int}_{1>\lmb_1 >\dots>\lmb_r>0}\ \prod_{j=1}^r s_j^{b} \prod_{1\leq j<k\leq r} ( s_j-s_k)^a \ d s_1 \wedge \dots \wedge d s_r=F(0),
\end{split}\nonumber
\end{equation}
where we performed in turn the change of variables $\lambda_j^2=t_j$ and $t_j=s_j/(1-s_j)$, and the last equality follows by \eqref{selberg}.
\end{proof}

%Observe that in general the dual of a \K\ form is not defined (see \cite[Example 1.3]{lmbdual}), we are going to  prove that $\sw_{\W,\mu}$ is a \K\ metric on $\C^{n+1}$ in Lemma \ref{lemswkahl} below.

\subsection{Holomorphic isometries between dual Cartan--Hartogs domains}\label{sectotdual}
Consider a totally geodesic complex immersion $f:\W' \f \W$ between HSSNCT. Identify  $\W'$ with its image $f\left(\W'\right)\subset\W$ and observe that $f$ trivially extends to an injective morphism $f:V' \f V$ of the associated HPJTS $V'$ and $V$ (see \cite[Prop. 2.2]{loiscala}). Hence, the map:
\begin{equation}\label{liftdual}
\wi f : V'\times \C \f V\times \C, \qquad\wi f (z,w) = \left(f(z),w\right),
\end{equation}
satisfies:
\begin{equation}\begin{split} \label{eqpullbw}
 \wi f ^* \omega^*_{\W,\mu}& = \frac{i}{2} \de\deb \log \left(N_{\Omega}^{\mu}(f(z),- \ov {f(z)}) + |w|^{2}\right) \\
 & = \frac{i}{2} \de\deb \log \left(N_{\Omega'}^{\mu}(z, - \ov  {z})+|w|^{2}\right) = {\w}^*_{\W',\mu}.
\end{split}\end{equation}
Let us identify $V\cong \C^n$ and $V'\cong \C ^m$, as in the beginning of this section, we just proved the following result:
\begin{prop}\label{propextdual}
Let $\W$ be an HSSNCT, then any totally geodesic complex immersion $f:\W' \f \W$ extends to the  \K\ embedding $\wi{f}: \left(\C^{m+1}, {\w}^*_{\W',\mu}\right) \f \left(\C^{n+1}, {\w}^*_{\W,\mu}\right)$ to the corresponding duals Cartan--Hartogs domains, given by \eqref{liftdual}.
\end{prop}

As in Subsec. \ref{sectot}, the isotropy group $K= \operatorname{Aut}(V\cong\C^n, \left\{,,\right\})$ of $\operatorname{Aut}(\Omega)$, by \eqref{eqpullbw} induces a natural action by isometries of $K$ on $\left(\C^{n+1},\w^* _{\W,\mu}\right)$, given by:
\begin{equation}\label{eqautacdual}\begin{split} 
\tau \cdot \left(z,w\right) =  \left(\tau(z),w\right), \quad \tau\in\operatorname{Aut}(\Omega). 
\end{split}\end{equation}
Moreover as a consequence of Prop. \ref{propextdual} and of the Polydisc Theorem for HSSNCT (see \cite{helgason}), we can see a dual Cartan--Hartogs domain $\left(\C^{n+1},\w^* _{\W,\mu}\right)$ as a union of \K\ embedded {dual Hartogs--Polydisc} $M^*_{\Delta^r,\mu}= \left(\C^{r+1}, \w^*_{\Delta^r,\mu}\right)$ (see Ex. \ref{hartogspolydual}):
$$
\C^{n+1}= \cup_{\tau\in K}\, \tau \left(M^*_{\Delta^r,\mu}\right)
$$ 
where $r$ is the rank of $\W$  and $\Delta^r\subset \W$ is a $r$-dimensional complex polydisc totally geodesically embedded in $\Omega$.

\subsection{{Proof of Theorem \ref{thmmaindual}}}\label{darbouxsectiondual}
Let $M^*_{\Omega,\mu}=\left(\C^{n+1}, \w^*_{\W,\mu}\right)$ be an $n$-dimensional dual Cartan-Hartogs domain and $\left(\C^n , \left\{,,\right\}_\W\right)$ the HJPTS associated to $\W$. 
By Lemma \ref{lemswkahl}, $M^*_{\Omega,\mu}$ is a well-defined K\"ahler manifold. 
In order to prove the existence of global Darboux coordinates,
consider the map $\Phi_{\W,\mu}\!:\C^{n+1} \rightarrow \C^{n+1}$  given by:
\begin{equation}\label{eqPhi}
\Phi_{\W,\mu}(z,w)=\frac{1}{\sqrt{N^\mu_\Omega(z,- \ov z)+|w|^2}}\left(\sqrt{\mu N^\mu_\Omega(z,- \ov z)}\, \operatorname{B}_\W(z,-\ov z)^{-\frac{1}{4}}z, w\right)
\end{equation}
where $\operatorname{B}_\W$ and $N_\W$ are respectively the Bergman operator and the generic norm associated to $ \left\{,,\right\}_\W$. 
We show that $\Phi_{\W,\mu}$ satisfies:
\begin{enumerate}[{\rm (A$'$)}]

\item 
$
\Phi^*_{\W,\mu}\w_0 = \sw_{\W,\mu};
$

\item $\Phi_{\W,\mu}$ is a diffeomorphism with its image $\operatorname {Im} \left(\Phi_{\W,\mu}\right) $. % In particular $\Phi_{\W,\mu}\left(M^*_{\W,\mu}\right) = M_{\W,\mu}$ if and only if $\mu=1$;
\end{enumerate}

As in the proof of Theorem \ref{thmmain}, we start with the following two lemmata, where to shorten the notation we set $\Ns(z,\ov  z):=N^\mu_{\Omega}(z,-\ov  z)$.

\begin{lemma}\label{ddarbouxh}
Let  $f_\Omega\!:M_{\Omega,\mu}\rightarrow \mathds C^{d+1}$ be a smooth map of the form
$$
f_\Omega(z_1,\dots, z_n,w):=\frac{1}{\sqrt{\Ns(z,\ov z)+|w|^2}}(h_1(z),\dots, h_n(z),w)
$$
where $h:=(h_1,\dots, h_n)$ satisfies:
\begin{equation}\label{dcondition1}
 \partial\bar\partial \Ns(z,\ov z)=\sum_{j=1}^ndh_j(z)\wedge d\ov {h_j(z)},
\end{equation}
and 
 \begin{equation}\label{dcondition2}
\sum_{j=1}^d(h_jd\bar h_j-\bar h_jdh_j)=(\bar \partial-\partial) \Ns.
\end{equation}
Then
\begin{equation}\label{deqfpullb}\begin{split} 
\omega^*_{\W,\mu}=\frac{i}{2}f_\Omega^*\left(\sum_{j=1}^n dz\wedge d\bar z +  dw\wedge\bar d w\right).
\end{split}\end{equation}
\end{lemma}
\begin{proof}
The proof is totally similar to that of Lemma \ref{darbouxh}.
\end{proof}

\begin{lemma}\label{draffaello}
If $G\!:(\Omega, \omega_{hyp}) \rightarrow (\mathds C^n, w_0)$ is a holomorphic map satisfying $G^*\omega_0=\frac i2\partial\bar \partial \log \Ns$, and:
\begin{equation}\label{dla19}
\sum_{j=1}^n\left( G_jd\bar G_j-\bar G_jdG_j\right)=\bar \partial \log N_{\Omega^*}- \partial \log N_{\Omega^*},
\end{equation}
 then:
\begin{equation}\label{dcompare}
 \partial\bar\partial \Ns=\mu\sum_{j=1}^dd(N_{\Omega^*}^{\mu/2}G_j)\wedge d (N_{\Omega^*}^{\mu/2}\bar G_j),
\end{equation}
and
 \begin{equation}\label{dcondition2proven}
\mu N_{\Omega^*}^{\mu/2}\sum_{j=1}^d\left(G_j d (N_{\Omega^*}^{\mu/2}\bar G_j)-\bar G_j d (N_{\Omega^*}^{\mu/2} G_j)\right)=(\bar\partial- \partial) \Ns.
\end{equation}
\end{lemma}
\begin{proof}
The proof is totally similar to that of Lemma \ref{raffaello}.
\end{proof}

In \cite[Theorem 1.1]{loiscala}, A. Loi and A. Di Scala show that $G\!:\left(\C^n, \frac i2 \partial \bar \partial \log \Ns\right) \rightarrow (\mathds C^n, w_0)$ defined by:
$$
G(z)=\operatorname{B}_\Omega(z,- \ov z)^{-\frac{1}{4}} z, 
$$
is a global symplectomorphism. Thus, by Lemma \ref{ddarbouxh} and Lemma \ref{draffaello}, in order to prove (A$'$) we need only to check that such $G$ satisfies \eqref{dla19}. Also here the proof is very similar to that of (A), once substituting $F$ with $G$ and $\operatorname{D}(z,z)$ with $-\operatorname{D}(z,-z)$.

Following the same approach as in the proof of (B), we prove first that property (B$'$) holds for the Hartogs--polydisc case.
\begin{lemma}\label{biprimo}
Property {\rm(B$'$)} holds for the dual Hartogs--polydisc.
\end{lemma}
\begin{proof}
We apply \cite[Th. 1.1]{loizuddas}. Following the notation of  Example \ref{hartogspolydual}, we can write the K\"ahler potential $\varphi^*_{\Delta^n,\mu}(z,w)$ for the dual Hartogs--polydisc $M^*_{\Delta^n,\mu}$ as $\varphi^*_{\Delta^n,\mu}(z,w) = \wi\varphi^*_{\Delta^n,\mu}(|z_1|^2,\dots,|z_{n}|^2,|y|^2)$, where
$\wi\varphi^*_{\Delta^n,\mu}:\C^ {n+1} \f \R$ is given by:
\begin{equation}\label{eqphitd}\begin{split} 
 \wi\varphi^*_{\Delta^n,\mu}(x_1,\dots,x_{n},y):=\log \left(\prod_{j=1}^{n}\left(1+x_j\right)^{\mu}+y\right).
\end{split}\end{equation}
Then by \cite[Th. 1.1]{loizuddas}, the map:
{\small\begin{equation}\label{eqPhi0}\small 
\Phi_{\Delta^n,\mu}(z,w)=\frac{1}{\sqrt{  \prod_{j=1}^n(1+|z_j|^2)^\mu+|w|^2}}\left(\sqrt{\mu\prod_{j=1}^n(1+|z_j|^2)^{\mu}} \left(\frac{z_1}{\sqrt{1+|z_1|^2}},\dots,\frac{z_n}{\sqrt{1+|z_2|^2}}\right),{ w}\right),
\end{equation}}
is a diffeomorphism with its image if $\frac{\partial \wi\varphi_{(\Delta^n)^*,\mu}}{\partial x_{k}}>0$, $\frac{\partial \wi\varphi_{(\Delta^n)^*,\mu}}{\partial y}>0$. The two conditions are easily checked:
$$
\frac{\partial \wi \varphi_{(\Delta^n)^*,\mu}}{\partial x_j} =\frac{\mu\prod_{j=1}^n(1+x_j)^{\mu}}{(1+x_j)\left( \prod_{j=1}^n(1+x_j)^\mu+y\right)}>0,
$$
$$
\frac{\partial \wi \varphi_{(\Delta^n)^*,\mu}}{\partial y} =\frac{1}{ \prod_{j=1}^n(1+x_j)^\mu+y}>0,
$$
and property (B$'$) is verified for $\W= \Delta^n$.
\end{proof}
Proceeding now as in the proof of (B), the spectral decomposition of $\Phi_{\W,\mu}$ reads:
\begin{equation}\begin{split} \label{eqPhij}
 \Phi_{\W,\mu}(z,w)=\frac{1}{\sqrt{\prod_{j=1}^r \left(1+\lmb^2\right)^\mu+|w|^2}}\left(\sqrt{\mu \prod_{j=1}^r \left(1+\lmb^2\right)^\mu}\sum_{j=1}^{r} \frac{\lambda_{j}}{\left(1+\lambda_{j}^{2}\right)^{1 / 2}} c_{j}, w\right).
\end{split}\end{equation}
Comparing  \eqref{eqPhij} with \eqref{eqPhi0} and using Lemma \ref{biprimo}, we deduce that also $ \Phi_{\W,\mu}$ is a diffeomorphism (we apply \cite[Section 1.6]{dilr}), concluding the proof.

%{\color{red} \begin{remark}\rm\label{rmkphir1m1}
%For $r=1$ and $\mu=1$, $M_{\W,\mu}=\C H ^{n+1}$ and from 
%$$
%\Psi_{\Omega,\mu} (z,w)= \frac{1}{\sqrt{1+|z|^2+|w|^2}}(z,w),
%$$
%Namely the inverse map of  symplectic duality the hyperbolic space and its dual, defined by A. Di Scala and A. Loi in \cite{loiscala}. By Remark \ref{rmkpsir1m1} we see that for $r=1$ and $\mu=1$, we have $\Psi_{\Omega,\mu} (z,w)=\Phi_{\Omega,\mu} ^{-1} (z,w)$.
%\end{remark}
%}

\begin{remark}\label{rmkphi}\rm
The map $\Phi_{\Omega,\mu}$ enjoys the same properties as $\Psi_{\Omega,\mu}$. In particular, it is hereditary in the sense that
for any bounded symmetric domain $\Omega'\subset \C^m$  complex and totally geodesic embedded  $\Omega' \stackrel{f}{\hookrightarrow} \Omega$, such that $f(0)=0$, one has:
 $$
 \Phi_{\Omega',\mu}(z,w)={\Phi_{\Omega,\mu}}\left(\wi f(z,w)\right),
 $$
where $\tilde{f}: M^*_{\Omega^{\prime}, \mu} \rightarrow M^*_{\Omega, \mu}$ is the \K\ embedding given in \eqref{liftdual}. This can be proven as in Remark \ref{heredrem}, using Prop. \ref{propextdual} instead of Prop. \ref{propext}. Further, $\Phi_{\W,\mu}$ commutes with the holomorphic isometric action \eqref{eqautacdual} of the isotropy group $K\subset \operatorname{Aut}(\Omega)$ at the origin, i.e. $\Phi_{\W,\mu} \circ \tau=\tau \circ \Phi_{\W,\mu}$,
as it follows by:
\begin{equation*}\begin{split} 
 \Phi_{\Omega,\mu} \circ \tau \left(z,w\right)&=\frac{1}{\sqrt{N^\mu_\Omega(\tau(z),-\ov {\tau(z)})+|w|^2}}\left(\sqrt{\mu N^\mu_\Omega(\tau(z),-\ov {\tau(z)})}\,\operatorname{B}_\Omega(\tau(z),-\tau(z))^{-\frac{1}{4}}\tau(z), w\right)\\
 &= \frac{1}{\sqrt{N_{\Omega}^{\mu}(z, -\ov {z})+|w|^{2}}}\left(\sqrt{\mu N_{\Omega}^{\mu}(z,-\ov  {z})} \,\operatorname{B}_\Omega(z, -z)^{-\frac{1}{4}}\tau( z), w\right)\\
 & = \tau \circ \Phi_{\Omega,\mu}  \left(z,w\right).
\end{split}\end{equation*}

\end{remark}

\section{Proof of Theorem \ref{propsimpc} and Theorem \ref{mainduality}}\label{capacitysection}

A map $c$ from the class  ${\mathcal C} (2n)$ of all symplectic manifolds of dimension $2n$ to $[0, +\infty]$
is called a \emph{symplectic capacity} if it satisfies the following conditions (see e.g. \cite{hoferzhenerbook}):

\begin{itemize}
\item[-] ({\em monotonicity}) if there exists a symplectic embedding $(M_1, \omega_1)\rightarrow (M_2, \omega_2)$ then 
$c(M_1, \omega_1)\leq c(M_2, \omega_2)$; 

\item[-] ({\em conformality}) $c(M, \lambda\omega)=|\lambda|c(M, \omega)$, for every $\lambda\in\R\setminus \{0\}$; 

\item[-]  ({\em nontriviality}) $c(B^{2n}(1), \omega_0)=\pi =c(Z^{2n}(1), \omega_0)$.

\end{itemize}

Here $B^{2n}(1)$ and $Z^{2n}(1)$ are the open unit  ball  and the  open cylinder in the standard $(\R^{2n}, \omega_0)$, i.e.:
\begin{equation*}
B^{2 n}(r)=\left\{(x, y) \in \mathbb{R}^{2 n}\, \midddd\, \sum_{j=1}^{n} x_{j}^{2}+y_{j}^{2}<r^{2}\right\},
\end{equation*}
\begin{equation}\label{Zcil}
Z^{2n}(r)=\{(x, y)\in\R^{2n} \ | \ x_1^2+y_1^2<r^2\}.
\end{equation}

We begin computing the symplectic capacity for $\left(M_{\Omega,\mu}, \w_0 \right) $. The proof relies on the facts, pointed out in \cite{lmz}, that the unitary ball $(B^{2n}(1), \omega_0)$ can be embedded into $(\Omega,\omega_0)$ and the domain $(\Omega,\omega_0)$ can be embedded into $(Z^{2n}(1),\omega_0)$.

\begin{proof}[Proof of Theorem \ref{propsimpc}] Let $\W$ be an HSSNCT and let $\left(\C^n, \left\{,,\right\}_\W\right)$ be its associated HJPTS.  
We first prove that the unitary ball $(B^{2n+2}(1), \omega_0)$ can be   embedded into   $\left( M_{\Omega,\mu}, \omega_0\right)$ if $\mu\in(0,1]$.
Let $z=\lambda_{1} c_1+\cdots+\lambda_{r}c_r$  be the spectral decomposition of a regular point $z\in \Omega\subset\C^{n}$, then the distance $\di_0 (0,v)$ 
from the origin $0\in {\mathcal M}$ to $z$
 is given by 
\begin{equation}\label{flat distance}
\di_0( 0,z) = (z\mid z)^{\frac{1}{2}} = \sqrt{\sum_{j=1}^r\lmb^2_j},
\end{equation}
(see \cite[Proposition VI.3.6]{roos} for a proof). Since:
$$
1  \leq \sum_{j=1}^r\lmb^2_j + \prod _{j=1}^r\left(1- \lmb^2_j\right),
$$
and:
$$
|w|^2<N\left(z,z\right)^\mu = \prod_{j=1}^{n}\left(1-\lmb^{2}\right)^{\mu},
$$
it follows that
$$
(B^{2n+2}(1), \omega_0)\cap \C^{n}_{reg}\times \C
\subset (M_{\Omega,\mu}, \omega_0)\cap \C^{n}_{reg} \times \C, \qquad \mu\in \left(0,1\right].
$$
Since  the set of regular points $\C^{n}_{reg}$ of $\C^n$ is dense (\cite[Proposition IV.3.1]{roos}) and $\Omega=\{z  \mid  \nax{z}<1\}$ (see \cite[Corollary 3.15]{loos}) we get:  
\begin{equation}\label{ballintoomega}
\left(B^{2n+2}(1), \omega_0\right)\subset \left(M_{\Omega,\mu}, \omega_0\right), \qquad \mu\in \left(0,1\right],
\end{equation}
as wished.

Let now $Z^{2n}(1)=\{(x, y)\ | \ x_1^2+y_1^2<1\}$ be the unitary cylinder in $\R^{2n}$. In \cite[Section 5]{lmz} it is proved that the domain $\left( \Omega, \omega_0\right)$ can be embedded into  $(Z^{2n}(1), \omega_0)$. It follows immediately that $\left(M_{\Omega,\mu}, \omega_0\right)$ can be embedded into  $(Z^{2n+2}(1), \omega_0)$ for every $\mu>0$.

Thus the first equality follows by the monotonicity and by the nontriviality of a symplectic capacity. %that

Let us now compute the symplectic capacity of $\left(M^*_{\Omega,\mu},\w^*_{\W,\mu}\right)$. By Theorem \ref{thmmaindual} it follows:
$$
c\left(\C^{n+1},\w^*_{\W,\mu}\right)=c\left( \operatorname {Im}\left( \Phi_{\W,\mu}\right),\w_0\right).
$$
Therefore, it is enough to show that:
\begin{equation}\label{dadim}
\left\{
\begin{array}{l}
B^{2n+2}(\mu)\subset \operatorname {Im}\left( \Phi_{\W,\mu}\right) \subset Z^{2n+2} (\mu) \ \text{ if }\ \mu < 1 \\ 
B^{2n+2}(1)\subset \operatorname {Im}\left( \Phi_{\W,\mu}\right) \subset Z^{2n+2} (1) \ \text{ if }\ \mu \geq 1
\end{array}.
\right.
\end{equation}

Consider the expression of the symplectomorphism $\Phi_{\W,\mu}$ given in \eqref{eqPhij} in terms of spectral decomposition: 
$$
\Phi_{\W,\mu}(z,w) = \left(\sum_{j=1}^{r} \xi_j c_j, \, \xi_0 w \right),
$$
for
$$
\xi_j=\sqrt{\frac{{\mu \prod_{j=1}^r \left(1+\lmb_j^2\right)^\mu}}{{\prod_{j=1}^r \left(1+\lmb_j^2\right)^\mu+|w|^2}}}
 \frac{\lambda_{j}}{\sqrt{1+\lambda_{j}^{2}}}, \quad j=1,\dots,r,
$$
and
$$
\xi_0= \frac{1}{\sqrt{\prod_{j=1}^r \left(1+\lmb_j^2\right)^\mu+|w|^2}},
$$
where $z=\sum _{j=1}^r\lmb_j c_j$ is the spectral decomposition of $z\in \C^n$. Notice that:
$$
\xi_j^2 < \mu \quad j=1,\dots,n,\qquad \xi_0^2\left|w\right|^2< 1.
$$
Thus:
$$
\left\{
\begin{array}{l}
\operatorname {Im}\left( \Phi_{\W,\mu}\right) \subset Z^{2n+2} (\mu) \text{ if } \mu < 1 \\ 
\operatorname {Im}\left( \Phi_{\W,\mu}\right) \subset Z^{2n+2} (1) \text{ if } \mu \geq 1
\end{array}.
\right.
$$

It remains to show that: 
$$
\left\{
\begin{array}{l}
B^{2n+2}(\mu)\subset \operatorname {Im}\left( \Phi_{\W,\mu}\right) \ \text{ if }\ \mu < 1 \\ 
B^{2n+2}(1)\subset \operatorname {Im}\left( \Phi_{\W,\mu}\right)\ \text{ if }\ \mu \geq 1
\end{array}.
\right.
$$
Notice that  $\xi_0^2\left|w\right|^2$ can assume every value in the interval $\left[0,1\right)$. Assume that $\delta$ and $c$ are real positive constant such that
$$
\xi_0^2\left|w\right|^2=\delta^2\leq c^2<\min\left\{1,\mu\right\}.
$$
In order to prove that the sphere $S^{2n+2}(c)$ of $\C^{n+1}$ of radius $c$, is contained in $\operatorname {Im}\left( \Phi_{\W,\mu}\right)$, we need to show that the following system:
\begin{equation}\label{eqsys}
 \left\{
\begin{array}{l}
\xi_0^2\left|w\right|^2=\delta^2,  \\ 
\xi_j^2=x_j^2, \text{ for } j=1,\dots,n, 
\end{array}.
\right.
\end{equation}
has a solution in $\lmb_1,\dots,\lmb_r,|w|$, for any $\delta^2 < c^2$ and any $n$-uple $x_1,\dots,x_n\in \R$ such that $x_1^2+\dots+x_n^2=c^2-\delta^2$. We have:
$$
\xi_0^2\left|w\right|^2=\delta^2\  \Leftrightarrow\  |w|^2= \prod_{j=1}^r \left(1+\lmb_j^2\right)^\mu \frac{\delta^2}{1-\delta^2}.
$$
Substituting the last term of the previous equality in $\xi_j^2=x_j^2$ we get:
$$
\mu(1-\delta^2)\frac{\lmb_j^2}{1-\lmb_j^2}=x_j^2.
$$
Notice that the left hand side of the previous equation assume any value in the interval $\left[0,\mu(1-\delta^2)\right)$. Hence if $c^2-\delta^2< \mu(1-\delta^2)$, the system \eqref{eqsys} has a solution. By hypothesis $c^2<\min\left\{1,\mu\right\}$, hence if $\mu<1$ we have:
$$
c^2-\delta^2\leq \mu-\delta^2 < \mu(1-\delta^2),
$$
while if $\mu\geq 1$:
$$
c^2-\delta^2< 1-\delta^2 \leq \mu(1-\delta^2).
$$
Thus we get \eqref{dadim} and
conclusions follow by the monotonicity, conformality and nontriviality of a symplectic capacity.
\end{proof}

%As immediate corollary we get:
%\begin{cor}\label{thmdualany}
%There not exists a symplectic duality map for $\left(M_{\W,\mu},\w_{\W,\mu}\right)$ if $\mu<1$.
% \end{cor}

Recall that given a HSSNT $\Omega$, with associated Bergman operator $\operatorname{B}_\Omega$, the map $\Xi_\W\!:\Omega\rightarrow \mathds C^n$, 
$$
\Xi_\W(z)= \operatorname{B}_\Omega(z,z)^{-\frac14}z,
$$
satisfies the analogous of properties as the map $\Psi_{\Omega,\mu}$ of Theorem \ref{thmmain} (see remarks \ref{heredrem} and \ref{autrem}) and in addition it is a symplectic duality between $\left(\W, \w_{hyp}\right)$ and its dual  $\left(\C^n, \w^*_{hyp}\right)$ (see \cite{loiscala}). Notice that, according with the definition of symplectic dual given in the introduction and \eqref{eqwhyp}, we have (see also \cite[(13)]{loiscala}):
\begin{equation}\label{eqhypd}\begin{split} 
\w^*_{hyp} = \frac{i}{2}\de\deb \log N_\Omega(z, - \bar{z}).
\end{split}\end{equation}

%In order to prove Theorem \ref{mainduality}, that it cannot exists a symplectic duality between $M_{\W,\mu}$ and $M^*_{\W,\mu}$ when $\mu\geq 1$ (the case $\mu<1$ is a corollary to next section) unless $\W=\C H^n$ and $\mu=1$. More precisely we prove the following:
%\begin{prop}\label{propdualmu1}
%Let $\mu\geq 1$, then there exists a symplectic duality between $M_{\W,\mu}$ and $M^*_{\W,\mu}$ if and only if  $\W=\C H^n$ and $\mu=1$, in this case the map $\Psi_{\W,\mu}$ of Theorem \ref{thmmain} realizes a symplectic duality.
%\end{prop}
%
%

We are now in the position of proving Theorem \ref{mainduality}:
\begin{proof}[{Proof of Theorem \ref{mainduality}}]

Assume $\W=\C {\rm H}^n$ and $\mu=1$, it is immediate to check that $M_{\C {\rm H}^n,1}=\C {\rm H}^{n+1}$. Recall that the generic norm and the Bergmann operator for $\C {\rm H}^n$ are given by:
$$
N_{\C {\rm H}^k}(z,\ov z) = 1 - \left|z\right|^2\quad \text{ and }\quad \operatorname{B}_{\C {\rm H}^k}(z, \bar{z})^{-\frac{1}{4}} z = \frac{z}{\sqrt{1-\left|z\right|^2}},
$$ 
where $\left|z\right|^2= \sum_{j=1}^k |z_j|^2$. Substituting the previous expressions in \eqref{eqdualkf}, \eqref{eqhypd} and \eqref{eqPsi}, we see that $\left(\C^{n+1}, \w^*_{\C {\rm H}^n,1}\right)=\left(\C^{n+1}, \w^*_{hyp}\right)$ and that:
$$
\Psi_{\C {\rm H}^n,1}=\Xi_\W,
$$
we conclude by \cite[Theorem 1.1]{loiscala} that $\Psi_{\C H^n,1}$ is a symplectic duality.  Moreover, by substituting the previous expressions in \eqref{eqPhi} and by  \cite[Theorem 1.1]{loiscala} we see that 
$$
\Phi_{\C {\rm H}^n,1}\left(z\right)= \operatorname{B}_\Omega(z,-z)^{-\frac14}z = \Xi^{-1}_\W(z).
$$

Viceversa, when $\mu<1$ a symplectic duality does not exist due to Theorem \ref{propsimpc} while, for $\mu\geq 1$, if a symplectic duality between $(\C^{n+1},\omega^*_{\W,\mu})$ and $(M_{\W,\mu},\omega_0)$ exists, then $\operatorname {Vol}(\C^{n+1},\omega^*_{\W,\mu})=\operatorname {Vol}(M_{\W,\mu},\omega_0)$. In this case, by Lemma \ref{volume} and \eqref{lemvolmw0} we have:
\begin{equation}\label{mudual}
\frac{F(\mu)}{F(0)}=\frac{\mu^n}{n+1}.
\end{equation}
Since:
$$
\frac{F(\mu)}{F(0)}=\prod_{j=1}^r\frac{\Gamma(\mu+1+(j-1)\frac a2)\Gamma(b+2+(r+j-2)\frac a2)}{\Gamma(1+(j-1)\frac a2)\Gamma(\mu+b+2+(r+j-2)\frac a2)},
$$
when $\Omega$ is the complex hyperbolic space, $\mu=1$ is a solution to \eqref{mudual}. In fact, in this case $r=1$ and $b=n-1$, thus: 
$$
\frac{F(\mu)}{F(0)}=\frac{\Gamma(\mu+1)\Gamma(n+1)}{\Gamma(\mu+n+1)}=\frac{n!}{(\mu+n)\cdots(\mu+1)},
$$
is equal to $\mu^n/(n+1)$ if and only if $\mu=1$. We claim that:
\begin{equation}\label{gennaio}
\frac{F(1)}{F(0)}=\prod_{j=1}^r\frac{(1+(j-1)\frac a2)}{(b+2+(r+j-2)\frac a2)}\leq \frac1{n+1},
\end{equation}
and the equality holds if and only if $r=1$.

If the claim holds, since the left hand side of \eqref{mudual} is strictly decreasing in $\mu$ while the right hand side is strictly increasing, we can see that the only positive solution to \eqref{mudual} must lie in $\left(0,1\right)$, concluding the proof. 

In order to prove the claim recall that $n=r(b+1+\frac{a}{2}(r-1))$.
Thus, substituting $r=1$ in \eqref{gennaio}, which happens iff $\W=\C {\rm H} ^n$, one readily gets that the equality is verified.
To conclude, let us proceed by induction on $r$. Assume $r\geq 2$, by the inductive hypothesis we have:
\begin{equation*}\begin{split} 
 \prod_{j=1}^r & \frac{(1+(j-1)\frac a2)}{(b+2+(r+j-2)\frac a2)} \leq \frac{(1+(r-1)\frac a2)}{(b+2+(2r-2)\frac a2)}\cdot \frac1{r(b+1+\frac{a}{2}(r-1))}\\
 \leq &   \frac{1}{(b+2+(2r-2)\frac a2)}\cdot \frac1{r} =  \frac{1}{
r(b+1+\frac{a}{2}(r-1)) + r(1+\frac{a}{2}(r-1)))}\\
< & \frac{1}{
r(b+1+\frac{a}{2}(r-1)) + 1}=\frac1{n+1},
\end{split}\end{equation*}
which proves \eqref{gennaio} (notice that the last inequality is strict since we assumed $r\geq 2$).
\end{proof}

\end{document}